\documentclass[12pt]{amsart}
\usepackage{graphicx}
\usepackage{amssymb}
\usepackage[all]{xy}
\usepackage{amsmath}
\usepackage[french,english]{babel}
\usepackage{amssymb,amsmath,amsfonts,amsthm,amscd}
\usepackage{indentfirst,graphicx}
\usepackage[font={scriptsize},captionskip=5pt]{subfig}

\def\Projan{\mathop{\rm Projan}}
\newtheorem{theorem}{Theorem}[section]
\newtheorem{exa}[theorem]{Example}

\newtheorem{rem}[theorem]{Remark}

\newtheorem{cor}[theorem]{Corollary}

\newtheorem{tont}[theorem]{Definition}

\newtheorem{lem}[theorem]{Lemma}

\newtheorem{Prop}[theorem]{Proposition}
 \newtheorem{exmp}[theorem]{Example}

\newtheorem{thm-def}[theorem]{Theorem/Definition}
\newtheorem{defi}[theorem]{Definition}

\newtheorem{Lem}[theorem]{Lemma}

\let\:=\colon

\newcommand{\cJ}{{\mathcal J}}

\newcommand{\cO}{{\mathcal O}}

\def\Der#1#2{{\displaystyle { {d#1}\over {d#2}}}}
\def\part#1#2{{\partial#1\over\partial#2}}
\let\Sum=\sum \def\sum{\Sum\nolimits}

\def\AF{\hbox{\rm A$_F$\hskip 2pt}}
\def\Projan{\mathop{\rm Projan}}

\begin{document}

\title{ Infinitesimal bi-Lipschitz  Equivalence of Functions}
\author {Terence Gaffney}\thanks{T.~Gaffney was partially supported by PVE-CNPq Proc. 401565/2014-9} 
\begin{abstract}
We introduce two different notions of infinitesimal bi-Lipschitz equivalence for functions, one related to bi-Lipschitz triviality of families of functions, one related to homeomorphisms which are bi-Lipschitz on the fibers of the functions in the family. We show that the first is not a generic condition, and that the second is.\end{abstract}

\maketitle
\selectlanguage{english}

\section{Introduction}

In an earlier paper, \cite {G-L}, we developed the idea of infinitesimal bi-Lipschitz equivalence for families of hypersurfaces, and showed that this condition was generic. That  is, given a family of hypersurfaces with isolated singularities, which contained a smooth parameter space $Y$ as a stratum, then $Y$ contained a Zariski open set $U$, such that the condition held along $U$.

It is known that there are moduli for bi-Lipschitz right equivalence in families of functions. This was proved by Henry and Parusinski, \cite {H-P}, who showed that even in the case of two variables, a bi-Lipschitz invariant  related to the order of vanishing of $f$ along its polar curve could change  continuously in a family. On the positive side, Valette \cite{V} showed, that asking that the trivializing homeomorphisms have Lipschitz restriction to the fibers of the functions was a generic condition. 

In this paper we develop two notions of infinitesimal bi-Lipschitz equivalence, one related to each of the above equivalence relations, and show that the first is not generic while the second is. In examples showing the first fails to be generic the Henry-Parusinski invariant appears.

The approach of \cite {G-L} was based on the Lipschitz saturation of a ring, an idea due to Pham and  Teissier (cf. \cite{P}, \cite{PT}). Using this we defined the Lipschitz saturation of an ideal  and related this notion to the integral closure of a certain submodule associated to the family. Since whether or not a function is Lipschitz depends on the behavior of the function at pairs of points, it is not surprising that it is convenient in these questions to work on the product of a space with itself, and with the double of a function, a concept defined in section two. A curve on the product of a space with itself, is given by a pair of curves on the space.

Checking that a function $g$ is in the Lipschitz saturation of the Jacobian ideal of a function $f$, $f\:{\mathbb C}^n\to {\mathbb C}$ becomes a matter of checking that the double of $g$, denoted $g_D$, is in the double of the Jacobian ideal along pairs of curves on ${\mathbb C}^n$. In section 3 we give a criterion if $n=2$ for the inclusion to hold along pairs of polar curves.  Checking that $g$ is in the Lipschitz saturation of the Jacobian ideal of $f$ relative to $f$ means checking  that $g_D$ is in the double of the Jacobian ideal of $f$ along pairs of curves in the same fiber of $f$.

In the next section we review the background necessary to understand the saturation of an ideal, and the connection between saturation and integral closure.  We introduce our two notions of saturation. The invariant of Henry and Parusinski already shows the importance of the behavior of $f$ along its polar curves. Surprisingly, it is rare for a function to be in the saturation of its Jacobian ideal,  but we prove a lemma showing that it is in the saturation of the Jacobian ideal relative to $f$. 

 In section three we  define our two conditions, and give geometric interpretations of them. Given a vector field $v$ on $Y$, they imply the existence of a collection of vector fields on ${\mathbb C}^{n+k}$ each extending $v$, and defined on some open (but not Zariski open) subset of  ${\mathbb C}^{n+k}$ containing $Y$ in its closure, such that the field has the appropriate Lipschitz property.  

Section three concludes with the proof of genericity for the condition relative to $F$. 

The proof of the main theorem was done in May 2014 in preparation for a collaboration with Walter Neumann, Anne Pichon and Bernard Teissier supported by the CIRM through its ``Research in Pairs" program. Work on some of the other propositions and examples of this paper were deeply influenced by our continuing conversations.

It is a pleasure to acknowledge my debt to them.

\section{Background on Lipschitz saturation of an ideal}

Let  $(X, x)$ be a germ of a complex analytic space and $X$ a
small representative of the germ and let $\mathcal{O}_{X}$ denote the
structure sheaf on a complex analytic space $X$. One of the formulations of the definition of the infinitesimal Lipschitz condition uses the theory of
integral closure of modules, which we now review. This theory will also provide the tools for working with the conditions we define later.

\begin{defi} Suppose $(X, x)$ is the germ of a complex analytic space,
$M$ a submodule of $\mathcal{O}_{X,x}^{p}$. Then $h \in
\mathcal{O}_{X,x}^{p}$ is in the integral closure of $M$, denoted
$\overline{M}$, if for all analytic $\phi : (\mathbb{C}, 0) \to (X,
x)$, $h \circ \phi \in (\phi^{*}M)\mathcal{O}_{1}$. If $M$ is a
submodule of $N$ and $\overline{M} = \overline{N}$ we say that $M$
is a reduction of $N$.
\end{defi}

To check the definition it suffices to check along a finite number of curves whose generic point is in the Zariski open subset of $X$ along which $M$ has maximal rank. (Cf. \cite {G-2}.)

We will also need the notion of the Lipschitz saturation of an ideal and the relative saturation. The construction of these objects  has much in common with the formation of the integral closure of an ideal or module, for the construction of the integral closure of an ideal is an example of a general approach to constructing closure operations on sheaves of ideals and modules, given a closure operation on a sheaf of rings. Here is the idea. Denote the closure operation on the ring $R$ by $C(R)$. Given a ring, R, blow-up $R$ by an ideal $I$. (If we have a module $M$ which is a submodule of a free module $F$, form the blow-up $B_{\rho(\mathcal{M})}(\Projan \mathcal{R}(F))$, as in the last section.) Use the projection map of the blow-up to the base to pullback $I$ to the blow-up. Now apply the closure operation to the structure sheaf of the blow-up, and look at the sheaf of ideals generated by the pull back of $I$. The elements of the structure sheaf on the base which pull back to elements of the ideal sheaf are the elements of $C(I)$.

 Two examples of this are given by the normalization of a ring and the semi-normalization of a ring. (In the normalization, all of the bounded meromorphic functions become regular, while in the semi-normalization only those which are continuous become regular. Cf \cite{GV} for details on this construction.)  Consider $B_I(X)$, the blow-up of $X$ by $I$. If we pass to the normalization of the blow-up, then $h$ is in $\bar I$ iff and only if the pull back of $h$ to the normalization is in the ideal generated by the pullback of $I$ \cite{LT}. If we pass to the semi-normalization of the blow-up, then $h$ is in the weak sub-integral closure of $I$ denoted ${}^*I$, iff the pullback of $h$ to the semi-normalization is in the ideal generated by the pullback of $I$. (For a proof of this and more details on the weak subintegral closure cf. \cite {GV}).
 
 There is another way to look at the closure operation defined above; in the case of the integral closure of an ideal, we are looking at an open cover of the co-support of an ideal sheaf, and choosing locally bounded meromorphic functions on each open set, and seeing if we can write a regular function locally in terms of generators of the ideal using our locally bounded meromorphic functions as coefficients. This suggests, that in the Lipschitz case, we use locally bounded meromorphic functions which satisfy a Lipschitz condition. The closure operation on rings that this indicates is the Lipschitz saturation of a space, as developed by Pham-Teissier (\cite {PT}). 
 
In the approach of Pham-Teissier, let $A$ be a commutative local ring over ${\Bbb C}$, and $\bar A$ its normalization. (We can assume $A$ is the local ring of an analytic space $X$ at the origin in ${\Bbb C}^n$.) Let $I$ be the kernel of the inclusion
$$\bar A\otimes_{{\Bbb C}}\bar A\to \bar A\otimes_{A}\bar A.$$

In this construction, the tensor product is the analytic tensor product which has the right universal property for the category of analytic algebras, and which gives the analytic algebra for the analytic fiber product. 

Pham and Teissier then defined the Lipschitz saturation of $A$, denoted $\tilde A$, to consist of all elements $h\in \bar A$ such that $h\otimes 1-1\otimes h \in \bar A\otimes_{{\Bbb C}}\bar A$ is in the integral closure of $I$.  (For related results see \cite {L}.)

The connection between this notion and that of Lipschitz functions is as follows. If we pick generators $(z_1,\dots,z_n)$ of the maximal ideal of the local ring $A$, then $z_i\otimes 1-1\otimes z_i \in \bar A\otimes_{{\mathbb C}}\bar A$ give a set of generators of $I$. Choosing $z_i$ so that they are the restriction of coordinates on the ambient space,  the integral closure condition is equivalent to  
$$\|h(z_1,\dots,z_n)-h(z'_1,\dots,z'_n)\|\le C sup_i\|z_i-z'_i\|$$
holding on some neighborhood $U$, of $(0,0)$ on $X\times X$. This last inequality is what is meant by the meromorphic function $h$ being Lipschitz at the origin on $X$. (Note that the integral closure condition is equivalent to the inequality holding on a neighborhood $U$ for some $C$  for any set of generators of the maximal ideal of the local ring $A$. The constant $C$ and the neighborhood $U$ will depend on the choice.)

If $X,x$ is normal, then passing to the Lipschitz saturation doesn't add any functions. Denote the saturation of the blow-up by $SB_I(X)$, and the map to $X$ by $\pi_S$. Then we make the definition:

\begin{defi} let $I$ be an ideal in ${\mathcal{O}}_{X,x}$, then the {\bf Lipschitz saturation} of the ideal $I$, denoted $I_S$, is the ideal
$I_{S}=\{h\in{\mathcal{O}}_{X,x} \vert \pi^*_S(h)\in \pi^*_S(I)\}$.
\end{defi}

Since the normalization of a local ring $A$ contains the seminormalization of $A$, and the seminomalization contains the Lipschitz saturation of $A$, it follows that 
$\bar I\supset {}^*I\supset I_S\supset I$. In particular, if $I$ is integrally closed, all three sets are the same.

Here is a viewpoint on the Lipschitz saturation of an ideal $I$, which will be useful later. Given an ideal, $I$, and an element $h$ that we want to check for inclusion in $I_S$, we can consider $(B_I(X), \pi)$, $\pi^*(I)$ and $h\circ \pi$. Since $\pi^*(I)$ is locally principal, working at a point $z$ on the exceptional divisor $E$, we have a local generator $f\circ \pi$ of  $\pi^*(I)$. Consider the quotient $(h/f)\circ \pi$. Then $h\in I_S$ if and only if at the generic point of any component of $E$, $(h/f)\circ \pi$ is Lipschitz with respect to a system of local coordinates. If this holds we say $h\circ \pi\in (\pi^*(I))_S$.

We can also work on the normalized blow-up, $(NB_I(X), \pi_N)$. Then we say $h\circ \pi_N\in (\pi_N^*(I))_S$ if $(h/f)\circ \pi_N$  satisfies a Lipschitz condition at the generic point of each component of the exceptional divisor of $(NB_I(X), \pi_N)$ with respect to the pullback to $(NB_I(X), \pi_N)$ of a system of local coordinates on $B_I(X)$ at the corresponding points of $B_I(X)$. As usual, the inequalities at the level of $NB_I(X)$ can be pushed down and are equivalent to inequalities on a suitable collection of open sets on $X$. 

This definition can be given an equivalent statement using the theory of integral closure of modules. Since Lipschitz conditions depend on controlling functions at two different points as the points come together, we should look for a sheaf defined on $X\times X$.  We describe a way of moving from a sheaf of ideals on $X$ to a sheaf on $X\times X$.  Let $h\in \cO_{X,x}$; define $h_D$  in $ \cO^2_{X\times X,{x,x}}$, as $(h\circ \pi_1,h\circ\pi_2)$, $\pi_i$ the projection to the i-th factor of the product. Let $I$ be an ideal in $\cO_{X,x}$;  then $I_D$ is the submodule of $\cO^2_{X\times X,{x,x}}$ generated by the $h_D$ where $h$ is an element of $I$.

If $I$ is an ideal sheaf on a space $X$ then intuitively, $h\in \bar I$ if $h$ tends to zero as fast as the elements of $I$ do as you approach a zero of $I$. If $h_D$ is in ${\overline {I_D}}$ then the element defined by $(1,-1)\cdot (h\circ \pi_1,h\circ\pi_2)=h\circ \pi_1-h\circ\pi_2$ should be in the integral closure of the ideal generated by applying $(1,-1)$ to the generators of $I_D$, namely the ideal generated by $g\circ \pi_1-g\circ\pi_2$, $g$ any element of $I$. This implies the difference of $h$ at two points goes to zero as fast as the difference of elements of $I$ at the two points go to zero as the points approach each other. It is reasonable that elements in $I_S$ should have this property. In fact we have:

 \begin {theorem} Suppose $(X,x)$ is a complex analytic set germ, $I\subset\cO_{X,x}$. Then $h\in I_{S}$ if and only if 
$h_D\in{\overline I_D}$.
\end{theorem}
\begin {proof} This is theorem 2.3 of \cite {GL1}, and is proved there under the additional assumption that $h\in\bar I$. However, as we have noted if $h\in I_{S}$, then $h\in\bar I$. If $h_D\in{\overline I_D}$, it follows that $(1,0)\cdot h_D$ is in the integral closure of $\pi_1^*(I)$ on $X\times X$, which clearly implies $h\in\bar I$.
\end{proof}

Now we add the necessary structure to deal with families of functions.

Suppose $F\:{\mathbb{C}}^{n+k},0\to {\mathbb C}$, $F$ a family of functions parameterized by ${\mathbb C}^k$, $y$ coordinates on  ${\mathbb C}^k$, $z$ coordinates on   ${\mathbb C}^n$. Let $J_z(F)$ denote the ideal generated by the partial derivatives of $F$ with respect to the $z$ coordinates.

In \cite{G-L}, the notion of the relative Lipschitz saturation of an ideal was developed. Instead of using the product of the normalization of $B_{I}(X)$ with itself, one uses the fiber product of the normalization over $Y$. When we are working in the context of a family of spaces we will also use $I_S$ to denote this saturation. In a similar way, we can develop an equivalent integral closure condition using modules as before, just working on $X\times_Y X$ instead of $X\times X$. 

In practice we will be working with ideal sheaves on a family of spaces, where the ideals vanish on $Y$, and our local coordinates at points of $B_I(X^{n+k})$ consist of the pullbacks of a set of generators of $m_Y$ and local coordinates on the projective space(s) in the blow-up.  

It is not difficult to check that Theorem 2.3 of \cite{GL1} continues to hold in this new context.

If we consider ${\mathbb{C}}^{n+k}\times_Y{\mathbb{C}}^{n+k}$, then there is the hypersurface $X(F)$ defined by $F\circ p_1-F\circ p_2=0$. We can use this hypersurface to define the Lipschitz saturation of a coherent ideal sheaf relative to the fibers of $F$. Let  $X(F)_0$ denote the Z-open subset $(y, z_1,z_2)$ where neither $z_i=0$

\begin{defi} Suppose $F\:{\mathbb{C}}^{n+k},0\to {\mathbb C}$, $F$ a family of functions parameterized by ${\mathbb C}^k$, $I$ an ideal sheaf on $ {\mathbb{C}}^{n+k}$. 
Then the Lipschitz saturation of $I$ relative to $Y={\mathbb{C}}^{k}$ and $F$ and denoted $I_{S,F}$ consists of functions $h$ whose induced functions on 
$NB_{I}({\mathbb{C}}^{n+k})\times_YNB_{I}({\mathbb{C}}^{n+k})$ satisfy the necessary conditions for the inclusion of $h$ in $I_{S }$
when restricted to $NB_{I}({\mathbb{C}}^{n+k})\times_{Y,F}NB_{I}({\mathbb{C}}^{n+k})$.
\end{defi}
In this paper, $I$ will be either $J(f)$ (ie. $k=0$) or $J_z(F)$.
Notice that 
the closure of $X(F)_0$ in $NB_{J_z(F)}({\mathbb{C}}^{n+k})\times_YNB_{J_z(F)}({\mathbb{C}}^{n+k})$ is just 
$NB_{J_z(F)}({\mathbb{C}}^{n+k})\times_{Y,F}NB_{J_z(F)}({\mathbb{C}}^{n+k})$.

In Teissier's development of condition C, the first step in the proof of the idealistic Bertini Theorem is to show that for any element $f\in m_n$, then $f\in \overline {J(f)}$. This step is easy enough to be an exercise in \cite{T1}. It turns out that the analogue holds for $J(f)_{S,f} $ but not for 
$J(f)_{S} $. We first prove a lemma which does hold for both notions.

\begin{Lem} Suppose $f\:{\mathbb{C}}^{n},0\to {\mathbb C},0$, then for all $\Phi=(\phi_1,\phi_2)\: {\mathbb C},0\to {\mathbb{C}}^{n}\times {\mathbb{C}}^{n},(0,0)$, $tD(f_D\circ\Phi)\in\Phi^*(J_D(f))$, $t$ a coordinate on ${\mathcal O}^1$.
 
\end{Lem}
\begin{proof} Expand $tD(f_D\circ\Phi)$ using the chain rule. Then the assertion is true provided 
$$(0,\Sum \part f{z_i}\circ\phi_2(\Der{{x_2}_i}t t-\Der{{x_1}_i}t t))\in \Phi^*(J_D(f)).$$
Now, expanding ${x_j}_i$ in powers of $t$, we have ${x_j}_i=\Sum a^j_{il}t^l$, so 
$$\Der{{x_2}_i}t t-\Der{{x_1}_i}t t=\Sum la^2_{il}t^l-\Sum la^1_{il} t^l.$$
Meanwhile 
$${x_2}_i-{x_1}_i=\Sum a^2_{il}t^l-\Sum a^1_{lt} t^l.$$
So the order in $t$ of ${x_2}_i-{x_1}_i$ and $\Der{{x_2}_i}t t-\Der{{x_1}_i}t t$ are the same. This implies 
$(0,\part f{z_i}\circ\phi_2(\Der{{x_2}_i}t t-\Der{{x_1}_i}t t))\in \Phi^*(J_D(f))$
\end{proof}

Now we come to the analogue of  $f\in \overline {J(f)}$.

\begin{Prop} Suppose $f\:{\mathbb{C}}^{n},0\to {\mathbb C},0$, then $f\in J (f)_{S,f}$.
\end{Prop}
\begin{proof} Let $\Phi=(\phi_1,\phi_2)\: {\mathbb C},0\to {\mathbb{C}}^{n}\times_f {\mathbb{C}}^{n},(0,0)$. The hypothesis implies that the Taylor expansions of $f\circ \phi_i$ are the same, say $\Sum a_kt^k$. Then $tD(f_D\circ\Phi)\in\Phi^*(J_D(f))$ implies $(\Sum ka_kt^k,\Sum ka_kt^k)$ is in 
$\Phi^*(J_D(f))$ Suppose $l$ is the order in $t$ of $\Sum a_kt^k$. Then 

 $$tD(f_D\circ\Phi)\in\Phi^*(J_D(f))=(t^l,t^l)(\Sum  ka_kt^{k-l})$$ 
 Since $\Sum  ka_kt^{k-l}$ is a unit, $(t^l,t^l)\in \Phi^*(J_D(f))$ which implies the result. 
\end{proof}

The above proposition does hold for weighted homogenous germs, because for these germs we have $f\in J(f)$. But as the next example shows it does not even hold for  quasi-homogenous germs in two variables in general.

\begin{exmp} Let $f(x,y)=(1/3)x^3-y^7-xy^5$. Here wt$(x)=7$, wt$(y)=3$, wt$(f)=21$, wt$(xy^5)=22$. Consider the family of functions 
$F_t(x,y)=f(x,y)-txy^5$ This family is $\mu$ constant because the weight of $f_t$ is greater than the weight of $f$, hence the zero sets are equisingular, and the the family of functions has a rugose trivialization. However, $f\notin J(f)_S$.\end{exmp}

Here are the details of the calculation for the example. The polar curve of $f$ defined by  $f_x=0$ is given by $x=\pm y^{5/2}$; we define $\Phi$ using the two branches of this polar curve to be $\Phi(t)=[(t^5,t^2),(-t^5,t^2)]$. Then $\Phi^*(J(f)_D)$ has two generators $(f_y\circ\phi_1,f_y\circ\phi_2)$ and $(0,(f_y\circ \phi_2)(x\circ\phi_1-x\circ\phi_2))$. This follows because $(f_x)_D\circ\Phi=(0,0)$, and $(y\circ\phi_1-y\circ\phi_2)=0$.
 The inclusion $f_D\circ\Phi\in\Phi^*(J(f)_D)$ is equivalent to 
 $$f\circ\phi_2-f_y\circ\phi_2({{f\circ\phi_1}\over{f_y\circ\phi_1}})\in (f_y\circ\phi_2)(x\circ\phi_1-x\circ\phi_2)$$ 
In turn this is equivalent to:

$$-t^{14}+(2/3)t^{15}-(-t^{14}-(2/3)t^{15}){(1-5/7t)\over{(1+5/7t)}}\in (t^{12})(t^5)$$

$$t^{15}\in(t^{17}),$$
which is false.

\section{Infinitesimal Lipschitz Equisingularity Conditions for Functions} 

In this section we define two infinitesimal Lipschitz equisingularity conditions based on the two definitions of the previous section. We show the first is not a generic condition while the second is.

\begin{defi} A family of functions $F\:{\mathbb{C}}^{n+k},0\to {\mathbb C}$, $F$  parameterized by ${\mathbb C}^k$ is infinitesimally Lipschitz equisingular at the origin provided $$\part F{y_i}\in J_z(F)_S$$for all $y_i$. \end{defi}

\begin{defi} A family of functions $F\:{\mathbb{C}}^{n+k},0\to {\mathbb C}$, $F$  parameterized by ${\mathbb C}^k$ is infinitesimally fiber-wise Lipschitz equisingular at the origin provided $$\part F{y_i}\in J_z(F)_{S,F}$$for all $y_i$. \end{defi}

 Because both conditions can be phrased in integral closure terms, if a condition holds at a point, it holds at all points in a Z-open subset of $Y$.

\begin{exmp} Consider the family $F(t,x,y)=x^3-3t^2xy^4+y^6$ due to Henry and Parusinski. We show that for $t\ne 0$,$t$ fixed, $\part F{t}\notin
J(f_t)_S$.  Let $\Phi_t(y)=[t,(ty^2,y),(-ty^2,y)]$. As in the previous example, $\phi_i$ parametrize the two different branches of the polar curve defined by 
$\part {f_t}x=0$. For the inclusion  $\part F{t}\in
J(f_t)_S$ to hold we must have $$-6txy^4\circ\phi_2-\left((-6txy^4\circ \phi_1)\left({{\part {f_t}y\circ\phi_2}\over{\part{f_t}y\circ\phi_1}}\right)\right)\in ((\part {f_t}y\circ \phi_2)(x\circ\phi_1-x\circ\phi_2))$$
which is equivalent to
$$6t^2y^6\left(1+\left({{1+2t^3}\over{1-2t^3}}\right)\right)\in (ty^7)$$
which fails.

\end{exmp}

The curve used in the above example consists of the two branches of the polar curve defined by $(f_t)_x=0$.  The ratio $\left({{1+2t^3}\over{1-2t^3}}\right)$ is the ratio of the two invariants defined by Henry and Parusinski for this function.

In checking the inclusion on curves, it is along the branches of the polar curves that the submodule generated by the doubles of the partial derivatives has rank $<2$, so the inclusion of the definition is more stringent. It is not hard to check that if the equation for the polar has the form $f_x-cf_y=0$ then $f_t\in J(f)_S$ implies the restriction of $f_t/f_y$ to the polar is a Lipschitz function with respect to the coordinates on  ${\mathbb C}^2$.

It is not surprising that the condition fails already for fixed $y$. As the proof of our main result shows, if the condition holds for fixed $y$ for generic $y$, then it holds for the family for a perhaps smaller Z-open set of parameters.

The appearance of the Henry-Parusinski invariant is also not a surprise, as the next proposition shows. The set-up in \cite{H-P} is the initial term of $f$ is not zero at $(1,0)$, and the polar is given by $f_x=0$. The polar is parameterized by giving a fractional power series of $x$ in terms of $y$. Fix a line in the tangent cone of $f=0$ along which the tangent cone is singular ({\it an exceptional tangent line}), and fix a parameterization $\phi$ of a polar curve tangent to the line. Then the H-P invariants are obtained by considering  the initial terms of $\{f\circ \phi\}$ as $\phi$ varies through the branches, under the equivalence relation given by the ${\mathbb C}^*$ action on the  coordinate $y$. Denote the initial term of a function of $y$ by $()_{IN}$.
\begin{Prop} In the above setting, if $\phi_i$ are two branches of a polar curve defined by $\part fx -c\part fy=0$ tangent to an exceptional tangent line different from the $x$-axis, then the degree of the ratios $\displaystyle{{{f}\circ\phi_2}\over{{f}\circ\phi_1}}$ and $\displaystyle{{\part {f}y\circ\phi_2}\over{\part{f}y\circ\phi_1}}$ agree, and if either degree is $0$, the initial terms agree.

\end{Prop}
\begin{proof} By the Chain rule 
$$\Der {f\circ\phi_i}y=\part f x\circ\phi_i \Der {x\circ\phi_i}y+\part f y\circ\phi_i$$
$$=\part f y\circ\phi_i(c\Der {x\circ\phi_i}y+1)$$
The degree of the ratio of the leading terms of $f\circ\phi_i$, and  $\Der {f\circ\phi_i}y$ are the same; if the degree is $0$, then

$$
({{f\circ\phi_2}\over{f\circ\phi_1}})_{IN}=({{\Der {f\circ\phi_2}y}\over{\Der {f\circ\phi_1}y}})_{IN}$$

In turn the right hand side ratio equals:
$$({{\part f y\circ\phi_2(c\Der {x\circ\phi_2}y+1)}\over {\part f y\circ\phi_1(c\Der {x\circ\phi_1}y+1)}})_{IN}$$
Since both branches are tangent to the same line the constant parts of $\Der {x\circ\phi_i}y$ are the same, hence the terms $c\Der {x\circ\phi_i}y+1$ cancel.
\end{proof}
If the degree of the ratio in the theorem is non-zero, then it is necessary to multiply the leading term of $\displaystyle{{\part {f}y\circ\phi_2}\over{\part{f}y\circ\phi_1}}$ by $p_1/p_2$, where $p_i$ is the degree of $f$ on the $i$-th branch to recover the initial term of the ratio of $\displaystyle{{{f}\circ\phi_2}\over{{f}\circ\phi_1}}$ from $\displaystyle{{\part {f}y\circ\phi_2}\over{\part{f}y\circ\phi_1}}$.

As example 3.3 shows, the ratio  $\displaystyle{{\part {f_t}y\circ\phi_2}\over{\part{f_t}y\circ\phi_1}}$ appears in checking the inclusion of $\part Ft\in J(f_t)_D$ along pairs of polar branches.

Given a branch of a polar curve of the type we are considering, we can consider the intersection multiplicity of the branch with $f_y$. The branches of the polar curve can be organized into packets (cf. \cite{T2},\cite{B}) for details; the contact of  $f_y$ with the generic member of the i-th packet is an invariant denoted $e_i$. We let $C(B_1,B_2)$ denote the contact between two polar branches.

Again restricting to the case of a function of 2 variables, we can ask if $g(x,y)\in J(f)_S$. Given a pair of branches of a polar curve, with parameterizations $\phi_1$ and $\phi_2$, $\Phi=(\phi_1,\phi_2)$, we can consider the determinant $D(g,f)$ of the matrix whose columns are $g_D$ and $(\part fy)_D$ along $\Phi$. 

\begin{Prop} Suppose $f(x,y)$ defines a reduced plane curve, $\Phi=(\phi_1,\phi_2)$, a pair of branches of a polar curve of $f$, defined by $f_x-af_y=0$, $g\in{\mathcal O}_2$. 

i) If $D(g,f)\circ \Phi=0$, then $g_D\circ \Phi\in \Phi^*(J(f)_D)$ if and only if $i(g,B_i)\ge e_i$, $i=1,2$.

ii) If the order of $D(g,f)\circ \Phi=$min$\{i(g,B_i)+e_j\}$, then $g_D\circ \Phi\in \Phi^*(J(f)_D)$ if and only if $i(g,B_i)\ge e_i+C(B_1,B_2)$, $i=1,2$.

\end{Prop}
\begin{proof} We first prove ii). Following the approach in example 3.3, 
it is clear that $g_D\circ \Phi\in \Phi^*(J(f)_D)$ if and only if one of the inclusions 

$$g\circ \phi_2-{{\part {f}y\circ\phi_2}\over{\part{f}y\circ\phi_1}}g\circ \phi_1\in \Phi^*(\part {f}y\circ p_2(x\circ p_1-x\circ p_2))$$
$$g\circ \phi_1-{{\part {f}y\circ\phi_1}\over{\part{f}y\circ\phi_2}}g\circ \phi_2\in \Phi^*(\part {f}y\circ p_1(x\circ p_1-x\circ p_2))$$
holds.

These inclusions are both equivalent to  the order of $D(g,f)\circ \Phi$ being greater than or equal to the order of $(\part {f}y\circ \phi_1)(\part {f}y\circ\phi_2)(x\circ p_1-x\circ p_2)\circ \Phi$, so if one holds then both hold.

Suppose   the order of $D(g,f)\circ \Phi=i(g,B_1)+e_2$. Then $i(g,B_1)\ge e_1+C(B_1,B_2)$ if and only if the order of $D(g,f)\circ \Phi\ge e_1+e_2+C(B_1,B_2)$ if and only if the order of $D(g,f)\circ \Phi\ge $order of $(\part {f}y\circ \phi_1)(\part {f}y\circ\phi_2)(x\circ p_1-x\circ p_2)\circ \Phi$. Note that this shows that if either of these equivalent conditions hold then 
$i(g,B_2)+e_1\ge e_1+e_2+C(B_1,B_2)$, hence $i(g,B_2)\ge e_2+C(B_1,B_2)$.

Now we prove i); assume $D(g,f)\circ \Phi=0$.  In this case, $g_D\circ \Phi$ is a multiple of $\part fy_D\circ \Phi$ if $s\ne 0$. The condition $D(g,f)\circ \Phi=0$ implies $g_D\circ \Phi(s)={ {g_i\circ\phi_i(s)} \over{\part fy\circ\phi_i(s)} }\part fy_D\circ\Phi(s)$ for $s\ne 0$, $i=1,2$. The condition that $D(g,f)\circ \Phi=0$ implies that 
$${ {g_1\circ\phi_1(s)} \over{\part fy\circ\phi_1(s)} }={ {g_2\circ\phi_2(s)} \over{\part fy\circ\phi_2(s)} }$$
so either both ratios are smooth or neither is. Both ratios smooth is equivalent to $i(g,B_i)\ge e_i$, $i=1,2$. If neither ratio is smooth, then the order of $g\circ\phi_i $ is less than the order of $\part fy\circ\phi_i$, hence less than the order of $(\part fy\circ\phi_i, (\part fy\circ\phi_i)(x\circ\phi_1-x\circ\phi_2))$, which implies $g_D\circ \Phi\notin \Phi^*(J(f)_D)$
\end{proof}

We always have 
\begin{Prop} Suppose $f(x,y)$ defines a reduced plane curve, $\Phi=(\phi_1,\phi_2)$, a pair of branches of a polar curve of $f$, defined by $f_x-af_y=0$, $g\in{\mathcal O}_2$. 

Then,  $g_D\circ \Phi\in \Phi^*(J(f)_D)$ if and only if $i(g/f_y ,B_1\cup B_2)\ge C(B_1,B_2)$.
\end{Prop}
\begin{proof} $i(g/f_y ,B_1\cup B_2)\ge C(B_1,B_2)$ means that the order of vanishing of $(g/f_y\circ \phi_1-g/f_y\circ \phi_2)$ is greater than or equal to $C(B_1,B_2)$. This is easily seen to be equivalent to $D(g,f)\circ \Phi$ being greater than or equal to the order of $(\part {f}y\circ \phi_1)(\part {f}y\circ\phi_2)(x\circ p_1-x\circ p_2)\circ \Phi$.
\end{proof}

The previous two results mean that $g_D\circ \Phi\in \Phi^*(J(f)_D)$ is equivalent to either the restriction of $g$ to a polar curve is in $J(f)$, restricted to the polar, or the restriction of $g/\part fy$ to a polar curve is a lipschitz function. Saying that $g\in J(f)_S$ implies in general that the restriction of $g/\part f{z_i}$ to a polar curve is a lipschitz function, for appropriate $i$; for on the polar curve the ratios $T_j/T_i$ are constant, so the Lipschitz condition becomes the ordinary condition for a function to be Lipschitz.

For the family of examples of form $f=x^p+y^q, p<q$ there is a positive result using the last two propositions which permits the comparison of our condition with that  of Fernandes and Ruas. A family of functions is strongly bi-Lipschitz trivial if there exists a Lipschitz vector field  which integrates to give the trivialization. In \cite {FR}, they found a criterion for a family of quaishomogeneous functions to be strongly bi-Lipschitz trivial. In the two variable case given weights $p$ and $q$, $p<q$, the condition for $f+tg$ to be strongly bi-Lipschitz trivial, where $f$ is weighted homogeneous of weight $d$ is for the weight of $g$, denoted $w(g)$ to satisfy 
$$w(g)\ge d+q-p$$

\begin{exmp} If $f(x,y)=x^p+y^q$,$(p-1,q-1)=1$, then $x^iy^j\in J(f)_s$, for $(0, j)$, $j\ge q$ or for 
$i(q-1)+j(p-1)\ge (q-1)(p-1)+(q-1)$. 
\end{exmp}
The condition $(0, j)$, $j\ge q$ is true because $y^j\in J(f)$; $(p-1,q-1)=1$ implies the generic polar curves of $f$ are irreducible. Because the generic polars are irreducible, we parameterize pairs of polars by 
$$(x=c\xi_1y^{q-1/p-1},y), (x=c\xi_2y^{q-1/p-1},y)$$ where the $\xi_i$ are $p-1$ roots of unity. Notice that for the monomials $x^iy^j$, we can choose the roots of unity so that $\xi^i_1\ne \xi^i_2$. This ensures that $D(x^iy^j,f)\circ \Phi$ has the expected order, so ii) of the proposition above applies.
For the  pairs of generic polar curves then $C(B_1,B_2)+e_1=(q-1)/(p-1)+(q-1)$, so we require 
$$i((q-1)/(p-1))+j\ge (q-1)/(p-1)+(q-1).$$ 

Although this condition only checks the inclusion of  $x^iy^j\in J(f)_s$ along polar curves, lengthy but straightforward calculations show that this is sufficient for all curves.

Comparing the estimate above with that of \cite {FR}, If $i=1$ then we see both estimates give the same lower bound on $j$, $j\ge q-1$. If $j=0$, then our lower bound on $i$ is $p$, while the bound of \cite {FR} is $p+(q-p)/p$. The convex hull of the three points $(1, q-1), (p,0), (p+(q-p)/p)$ then forms a wedge; integer points on the wedge to the left of the upper boundary give monomials for which our condition holds and that of \cite{FR} fails.

 We now work to show that the second equisingularity condition is generic in a family. The method of proof is similar to the proof that the corresponding condition for sets is generic. In turn, we will follow the lines of the proof of the Idealistic Bertini Theorem given in \cite{T1} p591-598. We will see exactly why the condition on fibers comes in. 
 
\noindent {\it Set-up} Let $F:{\mathbb C^k} \times {\mathbb C^n}\:\to {\mathbb C}$, be a family of functions parameterized by ${\mathbb C^k}$, with coordinates $y$ on ${\mathbb C^k}$, and $z$ on  ${\mathbb C^n}$. Then $f_y(z)=F(y,z)$ is a member of the family. Because we will be working at the generic point in the family of functions we can assume that the singular locus of $F$ is ${\mathbb C^k}\times 0$ which we also denote by $Y$, and that $Y$ is the union of the critical points of the $f_y$. Furthermore, we can assume that the $\AF$ condition holds along $Y$. This last condition implies 
$$\part F{y_i}\in \overline{J_z(F)}$$for all $y_i$.

 Our method of proof will involve working on the normalized blow-up of ${\mathbb C}^{k+n}\times_{Y,F} {\mathbb C}^{k+n}\times {\mathbb{ P}}^1$ by the ideal sheaf induced from the submodule $J_z(F)_D$, denoting $NB_{(J_z(F))_D}( {\mathbb C}^{k+n}\times_{Y,F} {\mathbb C}^{k+n}\times {\mathbb{ P}}^1)$ by $N$. We need to check that on each component of the exceptional divisor that the pullback of the element induced from $ ({{\partial F}\over{\partial {y_i}}})_D$ to the normalized blowup is in the pullback of $(J_z(F))_D$. 
 
The first step is to work in the module setting for the condition showing  that condition holds in our set-up for all pairs $((y,x),(y,x')$ where not both $x$ and $x'$ are $0$. This will reduce the proof of the main theorem to considering components of the exceptional divisor of $N$ which project to $Y$. 

\begin{Prop} The co-supports of $(J_z(F))_D$ on   $\mathbb {C}^{k+n}\times_Y \mathbb{ C}^{k+n}$ consist of

1) $Y\times (0,0)$

2) $\Delta(\mathbb {C}^{k+n}\times_Y \mathbb {C}^{k+n})$

3) $(0\times_Y \mathbb {C}^{k+n})\cup (\mathbb {C}^{k+n}\times_Y 0)$

\end{Prop}
\begin{proof} Along these sets the rank of the module is at most $1$, so these sets are in the co-support. Off of them we can assume our  pair $((y,x),(y,x')$  satisfies $x\ne x'$, $Df(x)\ne 0$, $Df(x')\ne 0$. This implies the module has rank $2$ at $((y,x),(y,x')$.\end{proof}

We show that a stronger condition holds off of $Y\times (0,0)$. 

\begin{Prop} In the set-up of this section, at points of $\mathbb {C}^{k+n}\times_Y \mathbb {C}^{k+n}$ off  of $Y\times (0,0)$, $$(\part F{y_i})_D\in {\overline {J_z(F)_D}}$$for all $y_i$.
\end{Prop}
\begin{proof} We consider two cases: $((y,x),(y,x')$ with $x=x'$ and $x\ne 0$, and $((y,x),(y,x')$, $x= 0$, $x\ne x'$.

For generators of $J_z(F)_D$ we use $(\part F{z_i})_D$ and $p_2^*(J_z(F))(0,I_{\Delta})$, where $I_{\Delta}$ denotes the ideal of the diagonal of  $\mathbb {C}^{k+n}\times_Y \mathbb {C}^{k+n}$ .

In the first case,  $J_z(F)$ at $(y,x)$ is ${\mathcal O}_{n+k}$, since $x\ne 0$. This implies that $(J_z(F))_D$ at $((y,x),(y,x)$ contains  $({{\mathcal O}_{n+k}})_D$ which implies the result.

In the second case, let $\Phi=(\phi_1,\phi_2)$ be a curve on $\mathbb {C}^{k+n}\times_Y \mathbb {C}^{k+n}$, $\Phi(0)= ((y,0),(y,x')$. Since the $\AF$ condition holds at $(y,0)$, there exist $\xi_j(t)$ such that 
$D_z(F)\circ \phi_1(\xi_j(t))={{\partial F}\over{\partial {y_j}}}\circ \phi_1$.
Then $${{\partial F}\over{\partial {y_j}}}_D\circ \Phi - (D_z(F)_D\circ \Phi (t))(\xi_j(t))=(0, g(t))$$ for some $g(t)$.
Since $(y,x')$ is not in the co-support of $J_z(F)$, and $((y,x),(y,x')$ is not in the co-support of $I_{\Delta}$, it follows that $$\Phi^*(p_2^*(J_z(F))(0,I_{\Delta}))={\mathcal {O}}^1(0,1),$$ so $(0, g(t))\in \Phi^*(p_2^*(J_z(F)))(0,I_{\Delta})$ which finishes the proof.
\end{proof}

\begin{theorem} Suppose $F(y,z)$ is a family of functions in the set-up of this section. Then there exists a Zariski open subset $U$ of $Y$ such that at points of $U\times (0,0)$, 
$$\part F{y_i}\in J_z(F)_{S,F}$$for all $y_i$.\end{theorem}
\begin{proof}
We work on $N$ and  we need to check that on each component of the exceptional divisor that the pullback of the element induced from $ ({{\partial F}\over{\partial {y_i}}})_D$ to the normalized blowup is in the pullback of $(J_z(F))_D$ which we denote by $\cJ(F)_D$. Denote the projection  to $Y$ by $p$. By the previous lemmas we need only consider those components of the exceptional divisor which project to $Y$ under the map to $\mathbb {C}^{k+n}\times_{Y,F} \mathbb {C}^{k+n}$. Since we are working over a Zariski open subset of $Y$ we may assume that every such component maps surjectively onto $Y$. Since we are working on the normalization, we can work at a point $q$ of the exceptional divisor such that $E$ is smooth at  $q$, $N$ is smooth at $q$ and the projection to $Y$ is a submersion at $q$.  Thus, we can choose coordinates at $q$, $(y',u', x')$, such that $y_i'=y_i\circ p$, and $u'$ defines $E$ locally with reduced structure. The key point is that ${ {\partial u'}\over{\partial {y_i'}}}=0\hskip 3pt \forall y_i$.

Following the ideas of the proof below, we can choose $U$, so that if $h_D$ is in the integral closure of $J_z(F)_D$ on the fibers of $\mathbb {C}^{k+n}\times_{Y,F} \mathbb {C}^{k+n}$ over $Y$, then 
$h_D\in \overline{J_z(F)_D}$ along $U\times(0,0)$. For example, we know, by Proposition 2.6, that  $(f_y)_D\in {\overline{ (J(f_y)_D}}$ on fibers of $\mathbb {C}^{k+n}\times_{Y,F} \mathbb {C}^{k+n}$ over $Y$, which is equivalent to  $F_D\in {\overline{ (J_z(F)_D}}$ on the fibers. This implies $F_D\in {\overline{ (J_z(F)_D}}$ along $Y$ in  $\mathbb {C}^{k+n}\times_{Y,F} \mathbb {C}^{k+n}$.

Let $\pi_i$ denote the composition of $\pi$, the projection from $N$ to  $\mathbb {C}^{k+n}\times_{Y,F} \mathbb {C}^{k+n}\times {\mathbb P} ^1$ with the projection $p_i$ to the $i$-th factor of $\mathbb {C}^{k+n}\times_{Y,F} \mathbb {C}^{k+n}\times {\mathbb P} ^1$, $i=1, 2$.

We have that $F\circ p_1+sF\circ p_2=F\circ p_1(1+s)=F\circ p_2(1+s)$  on $ \mathbb {C}^{k+n}\times_{Y,F} \mathbb {C}^{k+n}\times{\mathbb P}^1$. Pull this back to $N$ by $\pi$. Since we can assume  $F_D\in \overline{J_z(F)_D}$ along $U\times(0,0)$, the order of vanishing of $(F\circ p_1+sF\circ p_2)\circ \pi$ along each component of the exceptional divisor is at least as great as the order of vanishing of $\cJ(F)_D$. 

This follows because by the choice of $q$ we know the order of vanishing of $\cJ(F)_D$ locally along the fiber of $E$ in the fiber of $N$ is the same as the order of vanishing of $\cJ(F)_D$ locally along $E$.

Now take the partial derivative of $(F\circ p_1+sF\circ p_2)\circ \pi$ with respect to $y'_i$ at $q$. We get by the chain rule:

$${{(F\circ p_1+sF\circ p_2)\circ \pi}\over{\partial {y'_i}}}={{\partial F}\over{\partial {y_i}}}
\circ \pi_1+s{{\partial F}\over{\partial {y_i}}}\circ \pi_2 + (\sum\limits_{j=1}^{n}{{\partial F}\over{\partial {z_j}}}\circ \pi_1 {{\partial {z_i\circ \pi_1}}\over{\partial {{y'}_i}}}$$
$$+s{{\partial F}\over{\partial {z_j}}}\circ \pi_2 {{\partial {z_j\circ \pi_2}}\over{\partial {{y'}_i}}})+\part{s}{{y'}_i}F\circ p_2\circ \pi.$$

Since $y'_i$ and $u'$ are independent coordinates, the order of vanishing in $u'$ of the left hand side is the same as the order of vanishing in $u'$ of $(F\circ p_1+sF\circ p_2)\circ \pi$

Unlike the case of hypersurfaces, there is a term involving the derivative of $s$. If the component of $E$ at $q$ surjects onto $Y\times 0\times {\mathbb P}^1$, then this term is $0$ as $y'_i$ and $s$ are independent. Otherwise it may be non-zero. However,  we are working on  $N$, hence $(F\circ p_1+sF\circ p_2)\circ \pi=(F\circ p_2\circ \pi)(1+s)$, further the order of vanishing in $u'$ of the partials of $(F\circ p_2\circ \pi)$ and $(1+s)$ is the same as the originals. This implies the order of vanishing of $\part{s}{{y'}_i}F\circ p_2\circ \pi$ is at least as great as $(F\circ p_2\circ \pi)(1+s)$, so this term is well-behaved.

Now we work to re-shape the rest of the  terms to prove the theorem. Notice that since $z_i$ all vanish along $Y$, $z_i\circ \pi_j$ all vanish along $E$ at $q$.
We have 

$$\part{ F}{ {y_i}}\circ \pi_1+s\part{ F}{ {y_i}}\circ \pi_2 =$$
$$-(\sum\limits_{j=1}^{n}(\part{ F}{z_j}\circ \pi_1) 
(\part{ {z_j\circ \pi_1}}{ {y'}_i})+s((\part{ F}{{z_j}}\circ \pi_2) (\part{ {z_j\circ \pi_1}}{ {y'}_i})$$
$$-(\part{ F}{ {z_j}}\circ \pi_2) \left[ \part{ {z_j\circ \pi_1}}{ {y'}_i}-\part{ {z_j\circ \pi_2}}{{y'}_i}\right]))-{{\partial(F\circ p_1+sF\circ p_2)\circ \pi}\over{\partial {y'_i}}}-\part{s}{{y'}_i}F\circ p_2\circ \pi.$$

We want to show that the terms on the right hand side in the above expression are in the ideal generated by the pullback of the ideal sheaf on  $ \mathbb {C}^{k+n}\times_{Y,F} \mathbb {C}^{k+n}\times{\mathbb P}^1$ induced by $J_z(F))_D$. For this we use the curve criterion. We use a test curve to show that the order of vanishing of $\part{ F}{ {y_i}}\circ \pi_1+s\part{ F}{ {y_i}}\circ \pi_2$ along a component is same as the order of vanishing of the ideal $(J_z(F))_D$. This will imply that  $\part{ F}{ {y_i}}\circ \pi_1+s\part{ F}{ {y_i}}\circ \pi_2$ is in the ideal along the component.  We can choose a curve $\tilde{\Phi}$ such that $\tilde{\Phi}$ is the lift of a curve $\Phi=(\psi, \phi_1, \phi_2)$, $\Phi:{\mathbb C}:\to {\mathbb P}^1\times  \mathbb {C}^{k+n}\times_{Y,F} \mathbb {C}^{k+n}$. Further $\tilde\Phi(0)$ is a smooth point of the component and the ambient space, $\tilde\Phi$ is transverse to the component so that $u'\circ \tilde\Phi=t$, where $t$ is a coordinate in the local ring of ${\mathbb {C}}$ at the origin. This implies that if an ideal is generated by  $u'^p$, that the pullback is generated by $t^p$. Since the pullback of the ideal $(J_z(F))_D$ is locally principal, we can choose $\tilde\Phi(0)$ so that $(J_z(F))_D$ is generated by a power of $u'$.

Then we have $$\tilde\Phi^*(\part{ F}{ {y_i}}\circ \pi_1+s\part{ F}{ {y_i}}\circ \pi_2)=$$

$$-(\sum\limits_{j=1}^{n}(\part{ F}{ {z_j}}\circ \pi_1\circ \tilde\phi_1) 
 (\part{ {z_j\circ \pi_1}}{ {y'}_i})\circ \tilde\phi_1+\psi_2/\psi_1((\part{F}{ {z_j}}\circ \pi_2)\circ \tilde\phi_2 (\part{ {z_j\circ \pi_1}}{ {y'}_i}\circ \tilde\phi_1$$
$$-(\part{ F}{ {z_j}}\circ \pi_2)\circ \tilde\phi_2 \left[ \part{ {z_j\circ \pi_1}}{ {y'}_i}\circ \tilde\phi_1-
\part{ {z_j\circ \pi_2}}{ {y'}_i}\tilde\phi_2\right]))+\dots,$$
where the dots take the place of those terms which already vanish to the desired order. 
The right hand side will clearly be in the ideal ${\Phi}^*(J_z(F))_D)$, provided the pullback of $(\part{ F}{ {z_j}}\circ \pi_2) ( \part{ {z_j\circ \pi_1}}{ {y'}_i}-
\part{ {z_j\circ \pi_2}}{ {y'}_i})$ is. However, by construction, since $y'$ and $u'$ are independent coordinates,  the order of 
$\part{ {z_j\circ \pi_1}}{ {y'}_i}-\part{ {z_j\circ \pi_2}}{ {y'}_i}$ in $u'$ will be the same as the order of $z_j\circ \pi_1 -z_j\circ \pi_2$. Hence the pullback of $(\part{F}{ {z_j}}\circ \pi_2) ( \part{ {z_j\circ \pi_1}}{ {y'}_i}-
\part{ {z_j\circ \pi_2}}{ {y'}_i})$ does vanish to the desired order in $t$, which finishes the proof.
\end{proof}

\end{document}